\documentclass[reqno,12pt]{amsart}
\usepackage[english]{babel}
\usepackage{amsmath,amsfonts,amsthm,amssymb,amsbsy,upref,color,graphicx,hyperref,enumerate,comment,tikz}
\usepackage[a4paper,margin=3.05truecm]{geometry}
\usepackage{bbm}
\usepackage{soul}
\usepackage[hyperpageref]{backref}

\DeclareMathOperator{\argmin}{argmin}

\def\eps{{\varepsilon}}

\def\O{\Omega}

\def\R{\mathbb{R}}

\def\HH{\mathcal{H}}
\def\M{\mathcal{M}}

\def\eps{\varepsilon}

\newcommand{\cp}{\mathrm{cap}}

\newcommand{\be}{\begin{equation}}
\newcommand{\ee}{\end{equation}}
\newcommand{\bib}[4]{\bibitem{#1}{\sc#2: }{\it#3. }{#4.}}

\numberwithin{equation}{section}
\theoremstyle{plain}

\newtheorem{theo}{Theorem}[section]

\newtheorem{prop}[theo]{Proposition}
\newtheorem{defi}[theo]{Definition}

\theoremstyle{definition}
\newtheorem{rema}[theo]{Remark}

\title[On the continuity of the Continuous Steiner Symmetrization]{On the continuity of the Continuous Steiner Symmetrization}

\author[G. Buttazzo]{Giuseppe Buttazzo}

\date{}

\begin{document}

\maketitle

\hfill{\it Dedicated to Roger Wets for his 85th birthday}

\begin{abstract}
Starting from the Brock's construction of Continuous Steiner Symmetrization of sets, the problem of modifying continuously a given domain up to obtain a ball, preserving its measure and with decreasing first eigenvalue of the Laplace operator, is considered. For a large class of cases it is shown this is possible, while the general question remains still open.
\end{abstract}

\textbf{Keywords:} Steiner symmetrization, shape optimization, torsional rigidity, first eigenvalue, $\gamma$-convergence.

\textbf{2010 Mathematics Subject Classification:} 49Q10, 35P15, 49R50, 49J45, 49R05.

\section{Introduction}\label{sintro}

The problem of {\it rounding} more and more a given set $\O\subset\R^d$, keeping fixed its measure and asymptotically reaching a ball of the same measure, enters in a number of problems and has been widely considered in the literature. More precisely, given a bounded open set $\O\subset\R^d$, the goal is to construct a family of domains $(\O_t)$, with $t\in[0,1]$, such that $\O_0=\O$, $\O_1=\O^*$ where $\O^*$ is a ball of the same measure as $\O$, and $|\O_t|=|\O|$ for all $t\in[0,1]$, where by $|\cdot|$ we denote the Lebesgue measure.

In addition, we require that the mapping $t\mapsto\O_t$ be {\it continuous} with respect to some suitable topology, and that the family $(\O_t)$ satisfy some monotonicity property that will be specified later.

We notice that, without the last monotonicity requirement, a very simple construction would provide a solution. Take indeed a set $\O$ and a point $x_0$ far enough from $\O$; denoting by $B(x_0,r)$ the ball of center $x_0$ and radius $r$ and by $\omega_d$ the Lebesgue measure of the unit ball in $\R^d$, the family
$$\O_t=(1-t)^{1/d}\O\cup B(x_0,r_t)\qquad\text{with }r_t=\bigg(\frac{t|\O|}{\omega_d}\bigg)^{1/d}$$
satisfies the measure constraint $|\O_t|=|\O|$, is such that $\O_0=\O$ and $\O_1=\O^*$, and is continuous in several useful topologies. An example of such a family $(\O_t)$ is illustrated in Figure \ref {fig1}.

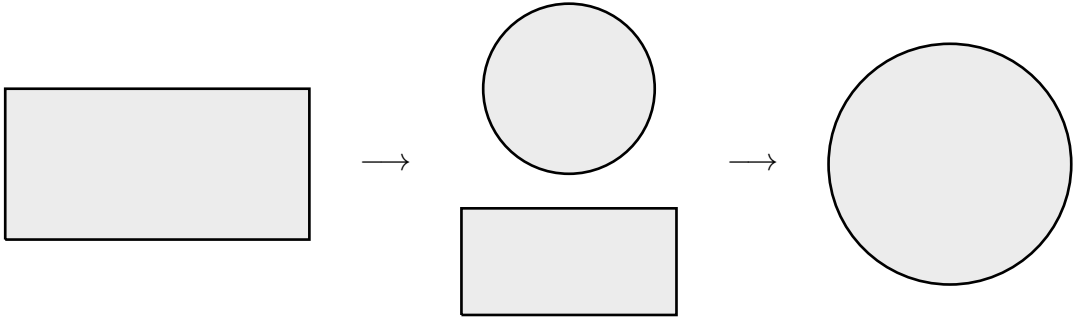
\begin{figure}[htbp]
\begin{tikzpicture}
\filldraw[fill=gray!15!white,line width=1](0,1)--(0,3)--(4,3)--(4,1)--(0,1)
(5,2) node {$\longrightarrow$}
(6,0)--(6,1.414)--(8.828,1.414)--(8.828,0)--(6,0)
(7.414,3) circle [radius=1.128cm]
(9.828,2) node {$\longrightarrow$}
(12.424,2) circle [radius=1.596cm];
\end{tikzpicture}
\caption{The sets $\O_0$, $\O_{1/2}$, $\O_1$ when $\O$ is the rectangle $]0,2[\times]0,1[$.}\label{fig1}
\end{figure}

The additional monotonicity conditions that we impose consists in the requirement that a suitable shape functional $F$ is monotone. For instance we could consider:
\begin{itemize}
\item[-]$F(\O)=P(\O)$, the {\it perimeter} in the sense of De Giorgi, and we require $P(\O)$ is nonincreasing;
\item[-]$F(\O)=\HH^{d-1}(\O)$, the {\it Hausdorff $d-1$ dimensional measure}, and we require $\HH^{d-1}(\O)$ is nonincreasing;
\item[-]$F(\O)=T(\O)$, the {\it torsional rigidity} defined below, and we require $T(\O)$ is nondecreasing;
\item[-]$F(\O)=\lambda(\O)$, the {\it first eigenvalue} of the Dirichlet Laplacian defined below, and we require $\lambda(\O)$ is nonincreasing;
\item[-]$F(\O)=h(\O)$, the {\it Cheeger constant}, and we require $h(\O)$ is nonincreasing.
\end{itemize}
In this paper we focus the attention mostly on the first eigenvalue $\lambda(\O)$ and on the torsional rigidity $T(\O)$.

More precisely, $\lambda(\O)$ is the first eigenvalue of the Laplace operator $-\Delta$ with Dirichlet conditions on $\partial\O$, that is the minimal value $\lambda$ such that the PDE
$$\begin{cases}
-\Delta u=\lambda u&\text{in }\O\,,\\
u\in H^1_0(\O),
\end{cases}$$
has a nonzero solution. Equivalently, by the min-max principle (see for instance \cite{he06}) $\lambda(\O)$ can be defined through the minimization of the Rayleigh quotient, as
$$\lambda(\O)=\min\bigg\{\Big[\int_\O|\nabla u|^2\,dx\Big]\Big[\int_\O u^2\,dx\Big]^{-1}\ :\ u\in H^1_0(\O),\ u\ne0\bigg\}.$$
An important bound for $\lambda(\O)$ is the {\it Faber-Krahn inequality} (see for instance~\cite{he06}, \cite{HP18})
$$\lambda(\O^*)\le\lambda(\O)\,,$$
which can be stated in a scaling free form as
$$|\O|^{2/d}\lambda(\O)\ge|B|^{2/d}\lambda(B),$$
where $B$ is any ball in $\R^d$.

The torsional rigidity $T(\O)$ is defined as $\int_\O u_\O\,dx$, where $u_\O$ is the unique solution of the PDE
$$\begin{cases}
-\Delta u=1&\text{in }\O\,,\\
u\in H^1_0(\O)\,,
\end{cases}$$
or equivalently through the maximization problem
$$T(\O)=\max\left\{\Big[\int_\O u\,dx\Big]^2\Big[\int_\O|\nabla u|^2\,dx\Big]^{-1}\ :\ u\in H^1_0(\O),\ u\ne0\right\}\,,$$
where the maximum is reached by $u_\O$ itself. Also for $T(\O)$ an important inequality holds, the {\it Saint-Venant inequality}
$$T(\O)\le T(\O^*)\,,$$
which can be stated in a scaling free form as
$$|\O|^{-(d+2)/d}T(\O)\le|B|^{-(d+2)/d}T(B)$$
where $B$ is any ball in $\R^d$. 

The monotonicity properties we require to the family $(\O_t)$ are then:
\begin{itemize}
\item[-]the mapping $t\mapsto\lambda(\O_t)$ is nonincreasing;
\item[-]the mapping $t\mapsto T(\O_t)$ is nondecreasing.
\end{itemize}

Concerning the continuity of the map $t\mapsto\O_t$ our requirement is that the solutions $u_t$ of the PDEs
$$\begin{cases}
-\Delta u_t=f&\text{in }\O_t\,,\\
u_t\in H^1_0(\O_t)\,,
\end{cases}$$
vary continuously in $t$ with respect to the strong $H^1(\R^d)$ convergence, for every right-hand side $f\in L^2(\R^d)$. This is the $\gamma$-convergence, that we describe more precisely in Section \ref{sgamma}.

When instead of a continuous family $(\O_t)$ we consider the discrete case of a sequence $(\O_n)$ such that
\begin{itemize}
\item[(i)]$\O_0=\O$, $|\O_n|=|\O|$ for every $n$, $\O_n\to\O^*$ in the $\gamma$-convergence,
\item[(ii)]$\lambda(\O_{n+1})\le\lambda(\O_n)$ and $T(\O_{n+1})\ge T(\O_n)$ for every $n$,
\end{itemize}
we have the problem that was first considered by Steiner, who proposed to use successive symmetrizations through different hyperplanes. More precisely, given a domain $\O\subset\R^d$ and a direction $\nu$, the {\it Steiner symmetrization} of $\O$ with respect to $\nu$ is defined as
$$\O^*_\nu=\bigg\{x\in\R^d\ :\ |x\cdot\nu|<\frac{\varphi\big(\pi(x)\big)}2\bigg\}\,.$$
Here $\pi(x)=x-\nu(x\cdot\nu)$ is the projection of a point $x\in\R^d$ on the hyperplane orthogonal to $\nu$ and, for each $y$ in this hyperplane,
$$\varphi(y)=\HH^1\big(\O\cap\pi^{-1}(y)\big)$$
is the length of the $y$-section of $\O$, where by $\HH^1$ we denote the 1-dimensional Hausdorff measure.

Note that the set $\O^*_\nu$ has the same volume of $\O$ and is symmetric with respect to the hyperplane orthogonal to $\nu$. In addition, it is well-known (see for instance~\cite{ALT91}) that the Steiner symmetrization decreases the first eigenvalue and increases the torsional rigidity, that is
$$\lambda(\O^*_\nu)\le\lambda(\O)\qquad\text{and}\qquad T(\O^*_\nu)\ge T(\O)\,.$$
By repeating this symmetrization procedure for a dense sequence of directions $\nu$, one obtains a sequence $\O_n$ of sets, all with the same measure, which $\gamma$-converge as $n\to\infty$ to the ball $\O^*$. 

The question is now to pass from the discrete Steiner symmetrization to a continuous one. Since successive Steiner symmetrizations allow to pass from a generic set to a ball, it is enough to construct a continuous family $\O_t$ of sets which transforms a set $\O$ into its Steiner symmetrization $\O^*_\nu$ for a fixed direction $\nu$. An explicit construction of a family $\O_t$ was proposed by Brock in \cite{bro95} (see also \cite{bro00}) and was called {\it Continuous Steiner Symmetrization}. We shortly recall the Brock's construction in Section \ref{sbrock}.

Unfortunately, the Brock's construction provides the $\gamma$-continuity of the family $\O_t$ only in very particular situations, as for instance when the initial domain $\O$ is convex, while in general discontinuities may occur, due to irregularities of the domains $\O_t$. On the other hand, the $\gamma$-continuity would be very useful in several situations, as for instance in the study of some Blaschke-Santal\'o diagrams, as illustrated in \cite{BP21}.

In the present paper we show that a modification of Brock's construction could be enough to provide the required $\gamma$-continuity of the family $\O_t$, at least for a larger class of domains $\O$. In \cite{BP21} a similar construction was made for polyhedral domains $\O$. Even if the arguments are not complete, we believe it could help to better understand the difficulties behind the Continuous Steiner Symmetrization.

In the last section we consider a possible alternative approach based on the De Giorgi theory of minimizing movements.

\section{The $\gamma$-convergence}\label{sgamma}

In this section we recall the definition and the main properties of $\gamma$-convergence; for all details, proofs, and generalization to the class of capacitary measures, we refer the interested reader to \cite{BB05}. For simplicity, we make the assumption that all the domains we consider are included in a given bounded open subset $D$ of $\R^d$, which is satisfied for the domains we consider later. In the following, for every domain $\O$, a function in $H^1_0(\O)$ is considered extended by zero on $\R^d\setminus\O$.

\begin{defi}\label{dgamma}
A sequence $(\O_n)$ of domains ia said to $\gamma$-converge to a domain $\O$ if for every $f\in L^2(\R^d)$ the solutions $u_{n,f}$ of the PDEs
$$\begin{cases}
-\Delta u=f&\text{in }\O_n\\
u\in H^1_0(\O_n)
\end{cases}$$
converge weakly in $H^1(\R^d)$ to the solution $u_f$ of the PDE
$$\begin{cases}
-\Delta u=f&\text{in }\O\\
u\in H^1_0(\O)\,.
\end{cases}$$
\end{defi}

We summarize here below the main properties of the $\gamma$-convergence. We refer to \cite{BB05} for all the details, properties, and proofs.

\begin{itemize}

\item The weak $H^1(\R^d)$ convergence of $u_{n,f}$ to $u_f$ is equivalent to the strong $H^1(\R^d)$ convergence. Indeed, integrating by parts we obtain
$$\int|\nabla u_{n,f}|^2dx=\int u_{n,f}f\,dx\to\int u_f f\,dx=\int|\nabla u_f|^2dx.$$

\item In the definition above it is not difficult to show that it is equivalent to require the weak $H^1(\R^d)$ convergence of $u_{n,f}$ to $u_f$ for every $f\in L^2(\R^d)$ or for every $f\in H^{-1}(\R^d)$. Indeed, if $f\in H^{-1}(\R^d)$ it is enough to approximate $f$ by a sequence $f_k\in L^2(\R^d)$, in the $H^{-1}$ norm, to obtain for every test function $\phi$
\[\begin{split}
&\left|\int\nabla u_{n,f}\nabla\phi\,dx-\int\nabla u_f\nabla\phi\,dx\right|=\left|\langle f,\phi\rangle_{H^1_0(\O_n)}-\langle f,\phi\rangle_{H^1_0(\O)}\right|\\
&\qquad\qquad\le\left|\langle f_k,\phi\rangle_{H^1_0(\O_n)}-\langle f_k,\phi\rangle_{H^1_0(\O)}\right|+\eps_k\|\phi\|\\
&\qquad\qquad=\left|\int\nabla u_{n,f_k}\nabla\phi\,dx-\int\nabla u_{f_k}\nabla\phi\,dx\right|+\eps_k\|\phi\|.
\end{split}\]
where $\eps_k\to0$. Passing to the limit first as $n\to\infty$ and then as $k\to\infty$ gives what claimed.

\item The $\gamma$-convergence can be defined in a similar way for {\it quasi-open} sets $\O\subset D$ or more generally for {\it capacitary measures} $\mu$ confined into $D$ (that is $\mu=+\infty$ outside $D$). Quasi-open sets are sets of positivity $\{u>0\}$ of functions $u\in H^1(\R^d)$, while capacitary measures are regular nonnegative Borel measures $\mu$ on $D$, possibly $+\infty$ valued, such that $\mu(E)=0$ for every Borel set $E\subset D$ with $\cp(E)=0$. For all details on quasi-open sets and capacitary measures we refer the interested reader to the book \cite{BB05}. Here we only recall that for a capacitary measure $\mu$ the corresponding PDE is formally written as
$$\begin{cases}
-\Delta u+\mu u=f&\text{in }D\\
u\in H^1_0(D)\cap L^2_\mu(D)
\end{cases}$$
and has to be intended in the weak sense, that is, $u\in H^1_0(D)\cap L^2_\mu(D)$ and
$$\int_D\nabla u\nabla\phi\,dx+\int_D u\phi\,d\mu=\langle f,\phi\rangle$$
for all $\phi\in H^1_0(D)\cap L^2_\mu(D)$. We notice that open sets or more generally quasi-open sets can be seen as capacitary measures: for a given domain $\O$ the capacitary measure representing it is the measure $\infty_{\O^c}$ defined as
$$\infty_{\O^c}(E)=\begin{cases}
0&\text{if }\cp(E\cap\O)=0\\
+\infty&\text{otherwise.}
\end{cases}$$

\item In Definition \ref{dgamma} it is possible to show (see Remark 4.3.10 of \cite{BB05}) that requiring the convergence of the solutions $u_n$ to $u$ for every right-hand side $f$ is equivalent to require the convergence $u_n\to u$ only for $f\equiv1$ and in the $L^2(D)$ sense. In particular, calling $u_\mu$ the unique solution of the PDE $-\Delta u+\mu u=1$ in $H^1_0(D)\cap L^2_\mu(D)$, the quantity
\be\label{gdist}
d_\gamma(\mu_1,\mu_2)=\|u_{\mu_1}-u_{\mu_2}\|_{L^2(D)}
\ee
is a distance on the space $\M$ of capacitary measures, which is equivalent to $\gamma$-convergence, and so $\M$ endowed with the distance $d_\gamma$ above is a compact metric space. Since the solutions $u_\mu$ are all equi-bounded (for instance they are all below by the solution $w$ of the Dirichlet problem $-\Delta w=1$ on $H^1_0(D)$, which is a bounded function) the $L^2$ norm in \eqref{gdist} can be replaced by any $L^p$ norm, with $1\le p<+\infty$. In particular, if $p=1$ and $\O_1\subset\O_2$ we have
$$\|u_{\O_1}-u_{\O_2}\|_{L^1}=\int u_{\O_2}dx-\int u_{\O_1}dx=T(\O_2)-T(\O_1),$$
and the $\gamma$-convergence is then reduced to the convergence of the corresponding torsional rigidities.

\item The first eigenvalue $\lambda(\O)$ (as well as all the other eigenvalues $\lambda_k(\O)$) and the torsional rigidity $T(\O)$ are continuous with respect to the $\gamma$-convergence.

\item The Lebesgue measure $|\O|$, or more generally integral functionals as $\int_\O f(x)\,dx$ with $f\ge0$ and measurable, are lower semicontinuous with respect to the $\gamma$-convergence on the domains $\O$.

\item As stated above, the space $\M$ of capacitary measures, endowed with the $\gamma$-convergence, is a compact metric space. On the contrary, the family of open sets (or also quasi-open sets) is not compact in $\M$; it is actually a dense subset of $\M$. The first example of a sequence of open sets $\O_n$ which $\gamma$-converges to a capacitary measure which is not a domain (actually to the Lebesgue measure) was given in \cite{ciomur}.

\item Several subclasses of $\M$ are dense with respect to the $\gamma$-convergence (see Proposition 4.3.7 and Remark 4.3.8 of \cite{BB05}). For instance:
\begin{enumerate}
\item[-] the class of measures $a(x)\,dx$ with $a\ge0$ and smooth;
\item[-] the class of smooth domains $\O\subset D$.
\item[-] the class of polyhedral domains $\O\subset D$;
\item[-] the class of measures of the form $a(x)\,d\HH^{d-1}$ with $a\ge0$ and smooth, where $\HH^{d-1}$ is the $d-1$ dimensional Hausdorff measure;
\item[-] the class of measures of the form $\HH^{d-1}\lfloor S$, where $S\subset D$ is a smooth $d-1$ manifold.
\end{enumerate}
\end{itemize}

\section{The Brock's construction}\label{sbrock}

We summarize rapidly here the construction by Brock (see \cite{bro95}, \cite{bro00}) of the continuous Steiner symmetrization, together with the properties important for our purpose. The first construction is for the unidimensional case; here taking the variable $t$ in $[0,+\infty]$ or in $[0,1]$ does not make any real difference.

\begin{itemize}

\item If $I$ is the interval $]a,b[$, then the continuous Steiner symmetrization $I^t$ is the interval $]a^t,b^t[$, where
$$a^t=\big(a-b+e^{-t}(a+b)\big)/2,\qquad b^t=\big(b-a+e^{-t}(a+b)\big)/2.$$

\item If $A$ is an open subset of $\R$ we consider the properties:
\begin{enumerate}
\item[(i)] $A(0)=A$;
\item[(ii)] if $I$ is an interval with $I\subset A(s)$, then $I^t\subset A(s+t)$ for every $t\ge0$.
\end{enumerate}
We define then the continuous Steiner symmetrization $A^t$ as
$$A^t=\bigcap\big\{A(t)\ :\ A(t)\text{ satisfies (i) and (ii)}\big\}.$$
In \cite{bro95} Brock proves that if $A$ is open then $A^t$ are open sets; in addition the monotonicity property
$$A\subset B \Longrightarrow A^t\subset B^t\text{ for every }t$$
holds.

\item Finally, if $A\subset\R$ is only measurable, we have
$$A=\bigcap_n A_n\setminus N$$
with $A_n$ open sets and $N$ Lebesgue negligible. We then define the continuous Steiner symmetrization $A^t$ of $A$ as
$$A^t=\bigcap_n A_n^t.$$
This definition is unique up to a nullset, and we still call continuous Steiner symmetrization a family $A^t$ such that $|A^t\triangle(\bigcap_n A_n^t)|=0$.

\end{itemize}

We can now pass to define the continuous Steiner symmetrization for subsets of $\R^d$, with respect to a hyperplane that, with no loss of generality, we can suppose to be $R^{d-1}$. For a general set $A$ we define the projection of $A$ on $\R^{d-1}$ as
$$A'=\big\{x'\in\R^{d-1}\ :\ (x',y)\in A\text{ for some }y\in\R\big\},$$
and for $x'\in A'$ the intersection of $A$ with $(x',\R)$ as
$$A(x')=\big\{y\in\R\ :\ (x',y)\in A\big\}.$$
Note that $A(x')$ is a one-dimensional set. When $A$ is an open subset of $\R^d$ we define its continuous Steiner symmetrization $A^t$ by
\be\label{cssd}
A^t=\big\{x=(x',y)\ :\ x'\in A',\ y\in(A(x'))^t\big\}.
\ee
If $A\subset\R^d$ is only measurable, we define its continuous Steiner symmetrization by the same formula as \eqref{cssd}, but up to Lebesgue negligible sets.

We stress that, for a bounded quasi-open set $A$, the previous construction only provides a measurable set defined up to a set of zero Lebesgue measure. In order to obtain that the symmetrized sets be still quasi-open and defined quasi-everywhere, it is convenient, for a bounded quasi-open set $A$, to define (by an abuse of notation) the symmetrized set $A^t$ in the following way: consider a decreasing sequence of bounded open sets $(A_n)$ with $\cp(A_n\setminus A)\to0$ and $A\subset A_n$. For any $t\in[0,1]$ the set $A_n^t$ is well defined, and by monotonicity we may define $A_n^t\supset A_{n+1}^t$. Then $(A_n^t)$ is $\gamma$-convergent and we define
\[A^t=\gamma-\lim_{n\to\infty}A_n^t.\]
In this way, the set $A^t$ is quasi-open. More details on this issue can be found in \cite{BB05}; in particular, the proofs that the construction above is independent of the sequence $A_n$ and that the Lebesgue measure is
preserved, are still missing.

The continuous Steiner symmetrization can be defined for any positive measurable function $u$ by symmetrizing its level sets:
\[\forall s>0\qquad\{u^t>s\}:=\{u>s\}^t.\]

The main properties of the Brock's construction are summarized here below, where $\lambda_k(\O)$ denotes the $k$-th eigenvalue of the Dirichlet Laplacian in $\O$.

\begin{prop}
For every bounded quasi-open set $\O\subset\R^d$ and every positive integer $k$ the mapping $t\mapsto\lambda_k(\O^t)$, is lower semicontinuous on the left and upper semicontinuous on the right.
\end{prop}

When the starting set $\O$ is convex, or more generally when the one-dimensional sections $\O(x')$ above are intervals, the $\gamma$-continuity actually occurs. However, this is not always the case, as the example of Figure \ref{Figdisc} shows. Up to the moment when the internal fracture appears the $\gamma$-continuity is verified; on the other hand, the Brock's construction removes the fracture instantaneously, and the $\gamma$-continuity is lost.

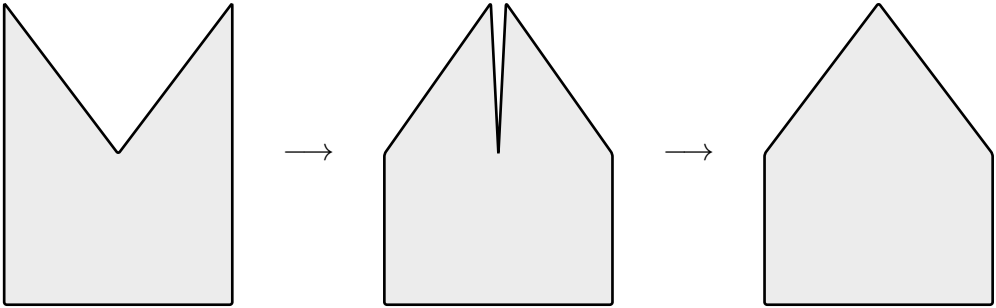
\begin{figure}[htbp]
\begin{tikzpicture}
\filldraw[fill=gray!15!white,line width=1,rounded corners=1pt]
(0,0)--(0,4)--(1.5,2)--(3,4)--(3,0)--cycle
(4.0,2) node {$\longrightarrow$}
(5,0)--(5,2)--(6.4,4)--(6.5,2)--(6.6,4)--(8,2)--(8,0)--cycle
(9.0,2) node {$\longrightarrow$}
(10,0)--(10,2)--(11.5,4)--(13,2)--(13,0)--cycle;
\end{tikzpicture}
\caption{A set $\O$ such that $t\mapsto\lambda(\O_t)$ is discontinuous.}\label{Figdisc}
\end{figure}

Since the torsional rigidity $T(\O_t)$ is increasing along the family $(\O_t)$, it has only countably many discontinuity points. Let $t_0$ be one of these points and assume that at $t_0$ we have two domains $\O^-,\O^+$ such that $\O^-\subset\O^+$ and
\be\label{o+o-}
\begin{cases}
T(\O_t)\to T(\O^-)\quad\hbox{as }t\to t_0\hbox{ from the left}\\
T(\O_t)\to T(\O^+)\quad\hbox{as }t\to t_0\hbox{ from the right}
\end{cases}\ee
In other words $\O^-$ is the domain with fractures, while $\O^+$ is the domain where the fractures have been removed.

\begin{rema}
In the one-dimensional case the existence of a $\gamma$-continuous family $(\O_t)$ cannot be obtained in general, since starting by $\O_0$ made of two segments and ending by $\O_1$ made of a single segment will necessarily produce a discontinuity of $T(\O_t)$ at some point $t_0$, independently of the construction of the family $(\O_t)$.
\end{rema}

In the case $d\ge2$ on the contrary, we can fill the discontinuity between $\O^-$ and $\O^+$ by constructing a $\gamma$-continuous family $(\O_t)$, with $\O_t$ increasing with respect to the set inclusion, and $\O_0=\O^-$, $\O_1=\O^+$.

\begin{theo}\label{01}
Let $d\ge2$ and let $\O_0\subset\O_1$ be two bounded open sets. Then there exists a $\gamma$-continuous family $\O_t$ of open sets ($t\in[0,1]$) such that
\be\label{monot}
\O_s\subset\O_t\qquad\text{for every }s<t.
\ee
\end{theo}

\begin{proof}
Let us denote by $C$ a large cube containing $\O_1$ and by $\Gamma(t)$ a Peano curve from $[0,1]$ onto $C$, that is a continuous mapping $\Gamma:[0,1]\to\R^d$ such that $\Gamma([0,1])=C$; we also choose $\Gamma(0)\in\O_0$. We define
$$\O_t=\big(\O_1\setminus\Gamma([0,1-t])\big)\cup\O_0\qquad\text{for every }t\in[0,1].$$
Note that $\O_t$ are open subsets of $\R^d$ and that for $t=0$ we obtain $\O_0$, while for $t=1$ we obtain $\O_1$. The family $\O_t$ above clearly satisfies the monotonicity property \eqref{monot}.

In order to show that the family $\O_t$ is $\gamma$-continuous, it is enough to prove that
$$\cp(\O_{t_n}\triangle\O_t)\to0\qquad\text{whenever }t_n\to t.$$
This comes from the fact that the mapping $\Gamma(t)$ is uniformly continuous, so that
$$|\Gamma(t)-\Gamma(t_n)|\le\omega(|t-t_n|)$$
for a suitable modulus of continuity $\omega$. Therefore $\O_t$ and $\O_{t_n}$ differ by a set which has a diameter less than $2\omega(|t-t_n|)$, hence of capacity which vanishes as $t_n\to t$.
\end{proof}

\begin{rema}
Since the proof of Theorem \ref{01} is only based on capacitary arguments, the same statement is valid in the more general case when $\O_0$ and $\O_1$ are quasi-open sets.
\end{rema}

\begin{rema}\label{poly}
When working with polyhedral domains (i.e. whose boundary is made of a finite number of subsets of hyperplanes) we are in the situation above. In fact, if $\O$ is a polyhedral domain, the Brock's construction provides a family $\O_t$ made of polyhedral domains, and we have a finite number of discontinuity points $t_1$, $t_2$, \dots $t_N$. In addition, for every discontinuity point $t_k$, the fracture $S$ is a $d-1$ dimensional polyhedral set, $\O^-=\O_{t_k}$ while $\O^+=\O_{t_k}\setminus S$, and then Theorem \ref{01} applies.
\end{rema}

In several situations (see for instance \cite{BP21}), thanks to the $\gamma$-density of polyhedral domains in the class of all domains, Remark \ref{poly} is sufficient to achieve the required goals. However, the question of existence of $\gamma$-continuous paths $(\O_t)$, with monotone $\lambda(\O_t)$ and $T(\O_t)$, between a general domain $\O_0$ and the ball $B$ with the same Lebesgue measure, remains.

Similar questions arise if, instead of the quantities $\lambda(\O_t)$ and $T(\O_t)$, one considers for instance the perimeter $P(\O_t)$, requiring the continuity of the map $t\mapsto P(\O_t)$ and its decreasing monotonicity.

The procedure of {\it removing fractures} mentioned after \eqref{o+o-} needs to be more rigorous. This can be made through the following result.

\begin{prop}
Let $\O_0$ be a given quasi open set and let $m\ge|\O_0|$. Then there exists a quasi open set $\hat\O$ solving the shape optimization problem
$$\min\big\{\lambda(\O)\ :\ \O_0\subset\O,\ |\O|\le m\big\}.$$
\end{prop}

\begin{proof}
The proof can be obtained directly by applying the existence result of \cite{BDM93}.
\end{proof}

In an analogous way we can obtain a solution for the shape optimization problem
$$\max\big\{T(\O)\ :\ \O_0\subset\O,\ |\O|\le m\big\}.$$
In particular, the case $m=|\O_0|$ is interesting; this allows to obtain, for every given $\O_0$, an optimal domain $\hat\O$ containing $\O_0$ and with the same measure as $\O_0$, which solves simultaneously the two shape optimization problems
$$\begin{cases}
\min\big\{\lambda(\O)\ :\ \O_0\subset\O,\ |\O|=|\O_0|\big\},\\
\max\big\{T(\O)\ :\ \O_0\subset\O,\ |\O|=|\O_0|\big\}.
\end{cases}$$
Indeed, if $\O_1$ is an optimal domain for the eigenvalue optimization problem and $\O_2$ an optimal domain for the torsion optimization problem, it is enough to take $\hat\O=\O_1\cup\O_2$.

In other words, if $\O_0$ is a Lipschitz domain, we have $\hat\O=\O_0$ while, in the case the set $\O_0$ presents some internal fractures, the set $\hat\O$ removes them.

\section{The minimizing movement approach}\label{sminmov}

An alternative approach to the Brock's construction of the family $\O_t$ through the Continuous Steiner Simmetrization could be the use of the De Giorgi minimizing movement theory, introduced in \cite{DG93} (see for instance \cite{A95}, \cite{AGS08} for a detailed presentation and further developments).

In our framework of shape functionals, the metric space $X$ could be the one of all measurable subsets $\O$ of the Euclidean space $\R^d$ with a prescribed Lebesgue measure, say $|\O|=1$, endowed with the $L^1$ distance
$$d(\O_1,\O_2)=|\O_1\triangle\O_2|.$$
Given a shape functional $F$ defined on $X$ one can consider the so-called {\it implicit Euler scheme} of time step $\eps$ and initial condition $\O_0$, which provides a discrete family $\O_{n,\eps}$ constructed recursively in the following way:
$$\O_{0,\eps}=\O_0,\qquad\O_{n+1,\eps}\in\argmin_{\O\in X}\Big\{F(\O)+\frac{1}{2\eps}|\O\triangle\O_{n,\eps}|^2\Big\}.$$
We may then set $\O_{t,\eps}=\O_{[t/\eps],\eps}$, where $[\cdot]$ stands for the integer part function, and say that $\O_t$ is a family of sets constructed by the minimizing movement procedure associated to the shape functional $F$ if for every $t\in[0,T]$ we have
$$|\O_t\triangle\O_{[t/\eps],\eps}|\to0\quad\text{as }\eps\to0.$$
If the limit above occurs only for a sequence $(\eps_n)$ (independent of $t$), we say that $\O_t$ is a generalized minimizing movement.

It is easy to see that the discrete sequence $\O_{n,\eps}$ is such that $F(\O_{n,\eps})$ decreases. It would be interesting to show, at least in the particular cases when the shape functional $F(\O)$ is the first eigenvalue $\lambda(\O)$, the opposite $-T(\O)$ of the torsional rigidity, or the perimeter $P(\O)$, or some convex combination of them, that the map $t\mapsto F(\O_t)$ is continuous and decreasing.

We do not know if the map $t\mapsto F(\O_t)$ above is continuous and decreasing, and the cases in which, as $t\to\infty$, the limit domain is a ball. Some results in this direction, in the case $F(\O)=P(\O)$ can be found in \cite{MPS22}, while some partial results in the case of spectral functionals can be found in \cite{MS22}.


\bigskip

\noindent{\bf Acknowledgments.} This work is part of the project 2017TEXA3H {\it``Gradient flows, Optimal Transport and Metric Measure Structures''} funded by the Italian Ministry of Research and University. The author is member of the Gruppo Nazionale per l'Analisi Matematica, la Probabilit\`a e le loro Applicazioni (GNAMPA) of the Istituto Nazionale di Alta Matematica (INdAM).

\bigskip
\small\noindent
Giuseppe Buttazzo: Dipartimento di Matematica, Universit\`a di Pisa\\
Largo B. Pontecorvo 5, 56127 Pisa - ITALY\\
{\tt giuseppe.buttazzo@unipi.it}\\
{\tt http://www.dm.unipi.it/pages/buttazzo/}

\end{document}